\documentclass[12pt]{amsart}
\usepackage{amsthm}
\usepackage{amssymb}
\usepackage{longtable}
\usepackage{geometry}
\usepackage{enumerate}
\usepackage{setspace}
\usepackage{pdflscape}
\usepackage{tikz}
\usepackage{tikz-cd}
\usepackage[unicode=true]
 {hyperref}
\hypersetup{
colorlinks=true,
urlcolor=black,
citecolor=blue,
linkcolor=blue,
}
\allowdisplaybreaks

\newtheorem{thm}{Theorem}[section]
\newtheorem{prop}[thm]{Proposition}
\newtheorem{lem}[thm]{Lemma}

\newtheorem{conjecture}[thm]{Conjecture}

\theoremstyle{definition}
\newtheorem{definition}[thm]{Definition}
\newtheorem*{notation}{Notation}

\theoremstyle{remark}
\newtheorem{remark}[thm]{Remark}

\numberwithin{equation}{section}

\newcommand{\ff}{\mathfrak{f}}
\newcommand{\fg}{\mathfrak{g}}
\newcommand{\fh}{\mathfrak{h}}
\newcommand{\Perm}{\mathrm{Perm}}
\newcommand{\Hom}{\mathrm{Hom}}
\newcommand{\Aut}{\mathrm{Aut}}

\newcommand{\Inn}{\mathrm{Inn}}
\newcommand{\Out}{\mathrm{Out}}
\newcommand{\Hol}{\mathrm{Hol}}
\newcommand{\res}{\mathrm{res}}
\newcommand{\ep}{\epsilon}
\newcommand{\GG}{\mathcal{G}}
\newcommand{\conj}{\mathrm{conj}}
\newcommand{\Norm}{\mathrm{Norm}}
\newcommand{\Cent}{\mathrm{Cent}}
\newcommand{\bZ}{\mathbb{Z}}

\begin{document}

\large 

\title[Hopf-Galois structures]{Hopf-Galois structures on finite extensions\\ with almost simple Galois group}

\author{Cindy (Sin Yi) Tsang}
\address{School of Mathematics (Zhuhai), Sun Yat-Sen University, Zhuhai, Guangdong, China}
\email{zengshy26@mail.sysu.edu.cn}
\urladdr{http://sites.google.com/site/cindysinyitsang/} 

\date{\today}

\maketitle

\vspace{-0.5mm}

\begin{abstract}In this paper, we study the Hopf-Galois structures on a finite Galois extension whose Galois group $G$ is an almost simple group in which the socle $A$ has prime index $p$. Each Hopf-Galois structure is associated to a group $N$ of the same order as $G$. We give necessary criteria on these $N$ in terms of their group-theoretic properties, and determine the number of Hopf-Galois structures associated to $A\times C_p$, where $C_p$ is the cyclic group of order $p$. 
\end{abstract}

\vspace{-0.5mm}

\tableofcontents

\section{Introduction}

Given a group $\Gamma$, write $\Perm(\Gamma)$ for its symmetric group, and recall that a subgroup $\mathcal{D}$ of $\Perm(\Gamma)$ is said to be \emph{regular} if the map
\[\label{xi map}
\xi_{\mathcal{D}}:\mathcal{D}\longrightarrow\Gamma;\hspace{1em}\xi_\mathcal{D}(\delta) = \delta(1_\Gamma)\]
is bijective. The images of the left and right regular representations
\[\begin{cases}
\lambda:\Gamma\longrightarrow \Perm(\Gamma);\hspace{1em}\lambda(\gamma) = (x\mapsto \gamma x)\\
\rho:\Gamma\longrightarrow\Perm(\Gamma);\hspace{1em}\rho(\gamma) = (x\mapsto x\gamma^{-1})
\end{cases}\]
of $\Gamma$ are examples of regular subgroups of $\Perm(\Gamma)$. Recall also that
\[\label{Hol1} \Hol(\Gamma) = \rho(\Gamma)\rtimes \Aut(\Gamma)\]
is the \emph{holomorph} of $\Gamma$. Alternatively, it is easy to check that
\[\label{Hol2} \Norm(\lambda(\Gamma)) = \Hol(\Gamma) = \Norm(\rho(\Gamma)),\]
where $\Norm(-)$ denotes the normalizer in $\Perm(\Gamma)$.

\vspace{1mm}

Given a finite Galois extension $L/K$ with Galois group $G$, by work of \cite{GP}, we know that the number of \emph{Hopf-Galois structures} on $L/K$ is equal to
\[ e(G) = \#\{\mbox{regular subgroups of $\Perm(G)$ normalized by $\lambda(G)$}\}.\]
In particular, for each group $N$ of the same order as $G$, the number of Hopf-Galois structures on $L/K$ of \emph{type} $N$ is equal to
\begin{equation}\label{e count} e(G,N) = \#\left\lbrace\begin{array}{c}\mbox{regular subgroups of $\Perm(G)$ which are}\\\mbox{isomorphic to $N$ and normalized by $\lambda(G)$}\end{array}\right\rbrace.\end{equation}
By \cite{By96}, this finer count may be calculated via the formula
\begin{equation}\label{B formula} e(G,N) = \frac{|\Aut(G)|}{|\Aut(N)|}\cdot \#\left\lbrace\begin{array}{c}\mbox{regular subgroups of $\Hol(N)$}\\\mbox{which are isomorphic to $G$}\end{array}\right\rbrace.\end{equation}
The computation of $e(G,N)$ has been a problem of interest; see \cite{Byott pq,Byott almost cyclic,Byott squarefree,Kohl98,Kohl ANT,Tsang PAMS} for some related work. We shall refer the reader to \cite[Chapter 2]{Childs book} for a more detailed discussion on Hopf-Galois structures. 

\vspace{1mm}

This paper is motivated by the case when $G$ is the symmetric group $S_n$ for $n\geq 5$. First, by \cite[Theorems 5 and 9]{Childs simple}, we know that
\begin{align}\label{Sn count1}
    e(S_n,S_n) & = 2 + 2\cdot\#\{\sigma\in A_n: \sigma\mbox{ has order $2$}\},\\\label{Sn count2}
    e(S_n,A_n\times C_2) & = 2\cdot\#\{\sigma\in S_n\setminus A_n: \sigma\mbox{ has order $2$}\},
\end{align}
where $A_n$ is the alternating group and $C_2$ is the cyclic group of order $2$. Also 

\noindent see \cite[Corollaries 6 and 10]{Childs simple}, which give explicit formulae for these two numbers. The case $n=6$ is slightly different because $S_6$ is not the full automorphism group of $A_6$, and as noted on \cite[p. 91]{Childs simple}, we have
\begin{equation}\label{S6 special}
e(S_6,\mathrm{PGL}_2(9))=0\mbox{ and }
e(S_6,\mathrm{M}_{10}) = 72,
\end{equation}
where $\mathrm{M}_{10}$ is the Mathieu group of degree $10$. Recently, the author has shown in \cite{Tsang Sn} that in fact
\begin{equation}\label{Sn count3} e(S_n,N) \neq 0 \mbox{ only if }N\simeq\begin{cases}
S_n, A_n\times C_2& \mbox{for $n\neq 6$},\\
S_6,A_6\times C_2,\mathrm{M}_{10},\mathrm{PGL}_2(9)&\mbox{for $n=6$}.
\end{cases}\end{equation}
Hence, the number $e(S_n,N)$ is known for every group $N$ of order $n!$.

\vspace{1mm}

Recall that a group $\Gamma$ is said to be \emph{almost simple} if 
\[ A \leq \Gamma \leq \Aut(A)\mbox{ for some non-abelian simple group $A$},\]
where $A$ is identified with its inner automorphism group $\Inn(A)$, and in this case $A$ is the socle of $\Gamma$. For $n\geq 5$, the symmetric group $S_n$ is almost simple with socle $A_n$ of index $2$. Note also that $\mathrm{PGL}_2(9)$ and $M_{10}$ are almost simple groups with socle $A_6$ of index $2$.

\vspace{1mm}

The purpose of this paper is to investigate to what extent the results (\ref{Sn count1}), (\ref{Sn count2}), and (\ref{Sn count3}) for the symmetric group may be generalized to an arbitrary finite almost simple group in which its socle has prime index.

\begin{notation}In the rest of this paper, assume that $G$ is a finite almost simple group with socle $A$ such that $A$ has prime index $p$ in $G$. Note that then
\begin{equation}\label{normal sub}
\mbox{$A$ is the unique non-trivial proper normal subgroup of $G$}.
\end{equation}
Also, we shall use the symbol $N$ to denote a group of the same order as $G$.
\end{notation}

For (\ref{Sn count1}), as shown in \cite[Theorem 1.3]{Tsang ASG}, we already know:

\begin{thm}\label{thm old}We have
\begin{align*} e(G,G)& = 2 + 2\cdot\#\{\sigma\in A:\sigma\mbox{ has order $p$}\}\\
& \hspace{4em} +2\cdot \frac{p-2}{p-1} \cdot \#\{\sigma\in G\setminus A:\sigma\mbox{ has order $p$}\},
\end{align*}
provided that $\Inn(G)$ is the only subgroup isomorphic to $G$ in $\Aut(G)$.
\end{thm}

For (\ref{Sn count2}), in Section~\ref{direct prod sec}, we shall prove:

\begin{thm}\label{thm1}We have
\[ e(G,A\times C_p) = 2\cdot \frac{1}{p-1}\cdot\#\{\sigma\in G\setminus A:\sigma\mbox{ has order $p$}\},\]
where $C_p$ is the cyclic group of order $p$. 
\end{thm}

Recall that a group $\Gamma$ is said to be \emph{perfect} if it equals its own commutator subgroup $[\Gamma,\Gamma]$, and \emph{quasisimple} if it is perfect and $\Gamma/Z(\Gamma)$ is simple, where $Z(\Gamma)$ is the center of $\Gamma$.

\vspace{1mm}

For (\ref{Sn count3}), we shall study the cases when $N$ is non-perfect and perfect separately. In Sections~\ref{case2 sec} and~\ref{case3 sec}, respectively, we shall prove:

\begin{thm}\label{thm2} If $N$ is non-perfect and $e(G,N)\neq0$, then $N\simeq A\times C_p$ or $N$ is an almost simple group with socle isomorphic to $A$.
\end{thm}

\begin{thm}\label{thm3}If $N$ is perfect and $e(G,N)\neq0$, then all of the conditions
\begin{enumerate}[(1)]
\item $N$ is a quasisimple group with $N/Z(N)$ isomorphic to $A$;
\item $A$ admits an automorphism having exactly $p$ fixed points;
\item $N/Z(N)$ has an element $\widetilde{\zeta}Z(N)$ of order $p$ such that
\[ \eta\widetilde{\zeta} \equiv \widetilde{\zeta}\eta\hspace{-2mm}\pmod{Z(N)}\mbox{ implies } \eta\widetilde{\zeta}=\widetilde{\zeta}\eta\mbox{ for all $\eta\in N$};\]
\end{enumerate}
hold, and in the case that $Z(N)$ is fixed pointwise by $\Aut(N)$, the condition
\begin{enumerate}[(1)]\setcounter{enumi}{+3}
\item $A$ has an element $\zeta$ of order $p$ such that 
\[\sigma\zeta = \zeta\sigma\mbox{ for some }\sigma\in G\setminus A;\]
\end{enumerate}
holds as well.
\end{thm}

For $n\geq 5$, it is known that $\Inn(S_n)$ is the only subgroup isomorphic to $S_n$ in $\Aut(S_n)$. This is because $\Aut(S_n) = \Inn(S_n)$ for $n\neq 6$ and was proven in \cite{S6} for $n=6$. Also, when $N = 2A_n$ is the double cover of $A_n$, condition (3) in Theorem~\ref{thm3} fails by the proof of \cite[Lemma 2.7]{Tsang Sn}. Hence, Theorems~\ref{thm old} to~\ref{thm3} imply the case when $G$ is $S_n$.

\vspace{1mm}

The converse of Theorem~\ref{thm2} is false by (\ref{S6 special}). By Theorem~\ref{thm1}, we have
\[ e(G,A\times C_p) \neq 0 \mbox{ if and only if $G$ splits over $A$ as a group extension}.\]
However, the author does not know whether there is any simple criterion on an almost simple $N$ with socle isomorphic to $A$ such that $e(G,N)\neq0$. Also, she does not know whether there exist any examples of $G$ and perfect $N$ for which all four conditions in Theorem \ref{thm3} are satisfied. It is possible that in fact $e(G,N)=0$ for all perfect $N$, but currently we are unable to prove this, and the conditions in Theorem \ref{thm3} might not be sufficient to rule out these $N$. But observe that if $e(G,N)\neq0$ with $N$ perfect, then $p$ divides the order of the Schur multiplier of $A$ by condition $(1)$ in Theorem~\ref{thm3}. Since $p$ divides the order of the outer automorphism group of $A$ by hypothesis, this already gives restrictions on $G$. We shall discuss more applications of our theorems in Section~\ref{app sec}.

\vspace{1mm}

Finally, let us make one remark. The following is due to N. P. Byott.

\begin{conjecture}\label{conj}Given any finite groups $\Gamma$ and $\Delta$ of the same order, if $\Gamma$ is insolvable and $e(\Gamma,\Delta)\neq0$, then $\Delta$ is also insolvable.
\end{conjecture}

It is known that Conjecture~\ref{conj} holds when $\Gamma$ is non-abelian simple \cite{Byott simple} and when $\Gamma$ is the double cover of $A_n$ for $n\geq 5$ \cite{Tsang HG}. Recently, it was also shown in \cite{Tsang solvable} that Conjecture~\ref{conj} holds when the order of $\Gamma$ and $\Delta$ is cubefree, less than $2000$, or satisfy some suitable conditions. Our Theorem~\ref{thm2} implies that Conjecture~\ref{conj} holds when $\Gamma$ is almost simple in which the socle has prime index. Let us remark that in the preprint \cite{Tsang QS}, the author has also extended the result of \cite{Byott simple} to the case when $\Gamma$ is quasisimple.

\section{Preliminaries}

In this section, let $\Gamma$ be a finite group.

\subsection{Regular subgroups in the holomorph}

Let $\Delta$ be a finite group, not necessarily of the same order as $\Gamma$. Let us recall some known methods which may be used to study regular subgroups of $\Hol(\Gamma)$.

\begin{definition} We have the following definitions.
\begin{enumerate}[$(1)$]
\item Given $\ff\in\Hom(\Delta,\Aut(\Gamma))$, a map $\fg$ from $\Delta$ to $\Gamma$ is said to be a \emph{crossed homomorphism with respect to $\ff$} if 
\begin{equation}\label{g prop}
\fg(\delta_1\delta_2) = \fg(\delta_1)\cdot\ff(\delta_1)(\fg(\delta_2))\mbox{ for all }\delta_1,\delta_2\in \Delta.
\end{equation}
Write $Z_\ff^1(\Delta,\Gamma)$ for the set of all such crossed homomorphisms.
\item Given $\varphi,\psi\in \Hom(\Delta,\Gamma)$, a \emph{fixed point} of $(\varphi,\psi)$ is an element $\delta\in\Delta$ such that $\varphi(\delta)=\psi(\delta)$, and $(\varphi,\psi)$ is said to be \emph{fixed point free} if it has no fixed point other than $1_\Delta$.
\end{enumerate}
\end{definition}

\begin{prop}\label{fg prop}The regular subgroups of $\Hol(\Gamma)$ isomorphic to $\Delta$ are precisely the subgroups of the shape
\[ \mathcal{D} = \{\rho(\fg(\delta))\cdot \ff(\delta):\delta\in \Delta\}\]
as $\ff$ ranges over $\Hom(\Delta,\Aut(\Gamma))$ and $\fg$ over the bijective maps in $Z_\ff^1(\Delta,\Gamma)$.
\end{prop}
\begin{proof}This follows directly from the definition that $\Hol(\Gamma) = \rho(\Gamma)\rtimes\Aut(\Gamma)$; or see \cite[Proposition 2.1]{Tsang HG} for a proof.
\end{proof}

\begin{prop}\label{h prop}Given $\ff\in\Hom(\Delta,\Aut(\Gamma))$ and $\fg\in Z_\ff^1(\Delta,\Gamma)$, define
\begin{equation}\label{h def}
\fh:\Delta\longrightarrow \Aut(\Gamma);\hspace{1em}\fh(\delta) = \conj(\fg(\delta))\cdot \ff(\delta),
\end{equation}
where $\conj(-) = \lambda(-)\rho(-)$. Then:
\begin{enumerate}[(a)]
\item The map $\fh$ is a homomorphism.
\item The fixed points of $(\ff,\fh)$ are precisely the elements of $\fg^{-1}(Z(\Gamma))$.
\item For all $\delta_1\in\ker(\ff)$ and $\delta_2\in \Delta$, we have $\fg(\delta_1\delta_2) =\fg(\delta_1)\fg(\delta_2)$.
\item For all $\delta_1\in\ker(\fh)$ and $\delta_2\in \Delta$, we have $\fg(\delta_1\delta_2) =\fg(\delta_2)\fg(\delta_1)$.
\end{enumerate}
\end{prop}
\begin{proof}See \cite[Proposition 3.4]{Tsang ASG} for (a) and the rest are easily verified. Let us just note that for (b), by definition $\delta\in\Delta$ is a fixed point of $(\ff,\fh)$ if and only if $\conj(\fg(\delta)) = \mathrm{Id}_\Gamma$, which is equivalent to $\fg(\delta)\in Z(\Gamma)$.
\end{proof}

Recall that a subgroup $\Lambda$ of $\Gamma$ is said to be \emph{characteristic} if $\varphi(\Lambda)=\Lambda$ for all $\varphi\in\Aut(\Gamma)$. In this case, clearly $\Lambda$ is normal in $\Gamma$, and
\[\Aut(\Gamma)\longrightarrow \Aut(\Gamma/\Lambda);\hspace{1em} \varphi \mapsto (x\Lambda\mapsto \varphi(x)\Lambda)\]
is a well-defined homomorphism.

\begin{prop}\label{char prop} Let $\Lambda$ be a characteristic subgroup of $\Gamma$. Given 
\[\ff\in\Hom(\Delta,\Aut(\Gamma))\mbox{ and }\fg\in Z_\ff^1(\Delta,\Gamma),\]
they induce two canonical maps
\[ \overline{\ff}_\Lambda: \Delta\longrightarrow\Aut(\Gamma)\longrightarrow\Aut(\Gamma/\Lambda)
\mbox{ and }\overline{\fg}_\Lambda : \Delta\longrightarrow \Gamma\longrightarrow\Gamma/\Lambda,\]
respectively, via compositions with the map $\Aut(\Gamma)\longrightarrow \Aut(\Gamma/\Lambda)$ above and the natural quotient map $\Gamma\longrightarrow\Gamma/\Lambda$. Then:
\begin{enumerate}[(a)]
\item We have $\overline{\ff}_\Lambda\in\Hom(\Delta,\Aut(\Gamma/\Lambda))$ and $\overline{\fg}_\Lambda\in Z_{\overline{\ff}_\Lambda}^1(\Delta,\Gamma/\Lambda)$. 
\item The subset $\fg^{-1}(\Lambda)$ is a subgroup of $\Delta$.
\item In the case that $\fg$ is bijective, there is a regular subgroup of $\Hol(\Lambda)$ which is isomorphic to $\fg^{-1}(\Lambda)$.
\end{enumerate}
\end{prop}
\begin{proof}Both (a) and (b) are clear; see \cite[Lemma 4.1]{Tsang HG} for a proof of (b). For part (c), see \cite[Proposition 3.3]{Tsang solvable}.
\end{proof}

Following the ideas in \cite{Byott simple} or \cite[Section 4]{Tsang HG}, we shall apply Proposition~\ref{char prop} to a maximal characteristic subgroup $\Lambda$ of $\Gamma$. In this case, the quotient $\Gamma/\Lambda$ is a finite non-trivial characteristically simple group, and we know that
\begin{equation}\label{N/M}
\Gamma/\Lambda \simeq T^m,\mbox{ where $T$ is a finite simple group and }m\in\mathbb{N}.
\end{equation}
This shall be a crucial step in the proof of Theorems~\ref{thm2} and~\ref{thm3}.

\subsection{Some group-theoretic facts}

We shall need the following basic properties of groups in which there is a normal copy of $A$ of index $p$.

\begin{lem}\label{copy of A lem}Assume that $\Gamma$ contains a normal subgroup $\Lambda$ isomorphic to $A$ and $[\Gamma:\Lambda]=p$. Then, either $\Gamma \simeq \Lambda\times C_p$ or $\Gamma$ is almost simple with socle $\Lambda$.
\end{lem}
\begin{proof}Since $\Lambda$ is normal in $\Gamma$, we have a homomorphism
\[ \Phi: \Gamma\longrightarrow \Aut(\Lambda);\hspace{1em}\Phi(\gamma) = (x\mapsto \gamma x \gamma^{-1}).\]
Put $C = \ker(\Phi)$, which is the centralizer of $\Lambda$ in $\Gamma$, and $C\cap\Lambda=1$ because $\Lambda$ has trivial center. If $C\neq1$, then since $[\Gamma:\Lambda]=p$, we deduce that
\[ \Gamma = \Lambda C = \Lambda \times C\mbox{ and }C\simeq C_p.\]
If $C=1$, then $\Gamma$ embeds into $\Aut(\Lambda)$ via $\Phi$, and since $\Phi(\Lambda)=\mathrm{Inn}(\Lambda)$, this implies that $\Gamma$ is almost simple with socle $\Lambda$.
\end{proof}

\begin{lem}\label{N lem}Assume that $\Gamma = A\times C_p$. Then:
\begin{enumerate}[(a)]
\item The non-trivial proper normal subgroups of $\Gamma$ are exactly $A$ and $C_p$.
\item The subgroups $A$ and $C_p$ of $\Gamma$ are characteristic.
\item We have $\Aut(\Gamma) = \Aut(A) \times \Aut(C_p)$.
\end{enumerate}
\end{lem}
\begin{proof}Let $\Lambda$ be any normal subgroup of $\Gamma$. Note that $\Lambda\cap A$ is normal in $A$. Since $A$ is simple, there are only two possibilities.
\begin{enumerate}[$\bullet$]
\item $\Lambda\cap A =A:$ Then $A\subset \Lambda$, so $\Lambda = A$ or $\Lambda = \Gamma$ since $A$ has prime index in $\Gamma$.
\item $\Lambda\cap A=1:$ Then $\Lambda$ has exponent dividing $p$. The projection of $\Lambda$ onto $A$, which is normal in $A$, hence cannot be $A$ and so must be trivial. It follows that $\Lambda\subset C_p$, so $\Lambda = 1$ or $\Lambda =C_p$.
\end{enumerate}
This proves (a), which in turn implies (b) and then (c).
\end{proof}

\begin{lem}\label{AS lem}Assume that $\Gamma$ is almost simple with socle $A$. Then:
\begin{enumerate}[(a)]
\item The center of $\Gamma$ is trivial;
\item The group $\Aut(\Gamma)$ embeds into $\Aut(A)$ via restriction to $A$.
\end{enumerate}
\end{lem}
\begin{proof}This is well-known; or see \cite[Lemmas 4.1 and 4.3]{Tsang ASG} for a proof.
\end{proof}

The next lemma gives some consequences of the classification of finite simple groups which we shall need.

\begin{lem}\label{CFSG lem}Assume that $\Gamma$ is non-abelian simple. Then:
\begin{enumerate}[(a)]
\item The outer automorphism $\Out(\Gamma)$ of $\Gamma$ is solvable.
\item Every $\varphi\in\Aut(\Gamma)$ has a fixed point other than $1_\Gamma$.
\item There is no subgroup isomorphic to $\Gamma$ in $\Aut(\Gamma)$ other than $\Inn(\Gamma)$.
\end{enumerate}
\end{lem}
\begin{proof}See \cite[Theorems 1.46 and 1.48]{G book} and \cite[Corollary 5.3]{Tsang ASG}.
\end{proof}

\section{The case when $N$ has a normal copy of $A$}\label{case1 sec}

In this section, assume that $N$ contains $A$ as a normal subgroup, in which case $[N:A]=p$ because $N$ is assumed to have the same order as $G$. Then, by Lemma~\ref{copy of A lem}, either $N\simeq A\times C_p$ or $N$ is almost simple with socle $A$. Let us prove an alternative formula for the number $e(G,N)$ which is similar but not quite the same as (\ref{B formula}). 

\subsection{A key observation} Let us first prove:

\begin{prop}\label{1st prop}A regular subgroup $\GG$ of $\Hol(N)$ isomorphic to $G$, which is not equal to $\lambda(N)$ or $\rho(N)$, is normalized by exactly one of $\lambda(N)$ and $\rho(N)$.
\end{prop}

Let $\GG$ be a regular subgroup of $\Hol(N)$ isomorphic to $G$ which is not equal to $\lambda(N)$ or $\rho(N)$. By Proposition~\ref{fg prop}, we know that
\[\GG = \{\rho(\fg(\sigma))\cdot \ff(\sigma): \sigma\in G\},\mbox{ where }
\begin{cases}\ff\in\Hom(G,\Aut(N)),\\
\fg\in Z_\ff^1(G,N)\mbox{ is bijective}.\end{cases}\]
We may also rewrite it as
\begin{equation}\label{GG2}  \GG = \{\lambda(\fg(\sigma))^{-1}\cdot \fh(\sigma): \sigma\in G\},\mbox{ where }\fh\in\Hom(G,\Aut(N))\end{equation}
is defined as in (\ref{h def}). Note that both $\ff$ and $\fh$ are non-trivial because 
\[ \begin{cases}
\GG\subset \rho(N) & \mbox{if $\ff$ were trivial},\\
\GG\subset\lambda(N)&\mbox{if $\fh$ were trivial},
\end{cases}\]
in which case we would have equality by the regularity of $\GG$. From (\ref{normal sub}), we then deduce that $\ker(\ff)$ and $\ker(\fh)$ are either trivial or equal to $A$.

\begin{lem}\label{fh injective}The following are true.
\begin{enumerate}[(a)]
\item If $\ff$ is injective, then $\GG$ is not normalized by $\rho(N)$.
\item If $\fh$ is injective, then $\GG$ is not normalized by $\lambda(N)$.
\end{enumerate}
\end{lem}
\begin{proof}Suppose that $\ff$ is injective. For any $\sigma\in G$ and $\eta\in N$, we have
\[ \rho(\eta)\cdot \rho(\fg(\sigma))\ff(\sigma)\cdot \rho(\eta)^{-1} = \rho(\eta\fg(\sigma)\ff(\sigma)(\eta)^{-1})\cdot \ff(\sigma).\]
By the injectivity of $\ff$, the above element lies in $\GG$ if and only if
\[ \eta\fg(\sigma)\ff(\sigma)(\eta)^{-1} = \fg(\sigma),\mbox{ or equivalently }\fh(\sigma)(\eta) = \eta.\]
But $\fh$ is non-trivial and so $\GG$ is not normalized by $\rho(G)$. This proves (a), and a similar argument using (\ref{GG2}) shows (b).
\end{proof}

Note that $A$ is characteristic in $N$. This is Lemma~\ref{N lem}(b) if $N\simeq A\times C_p$, and is because $A$ is the socle of $N$ if $N$ is almost simple. Hence, we have
\[ \overline{\ff}_A,\overline{\fh}_A\in\Hom(G,\Aut(N/A))\mbox{ and }\overline{\fg}_A\in Z_{\overline{\ff}_A}^1(G,N/A)\]
defined as in Proposition~\ref{char prop} and (\ref{h def}). Note that
\[ \Aut(N/A) \simeq \Aut(C_p) \simeq C_{p-1} \mbox{ (cyclic group of order $p-1$)}.\]
This, together with (\ref{normal sub}), implies that $\overline{\ff}_A$ is trivial, and so $\overline{\fg}_A$ is a homomorphism by Proposition~\ref{h prop}(c). But $N/A\simeq C_p$, and $\overline{\fg}_A$ is surjective because $\fg$ is bijective. Again from (\ref{normal sub}), we see that $\ker(\overline{\fg}_A)=A$, which gives $\fg(A)=A$. This equality shall be important in the arguments that follow. Note that $\overline{\fh}_A$ is trivial similarly by (\ref{normal sub}).

\vspace{1mm}

For any $\sigma\in G$ and $\eta\in N$, since $\overline{\ff}_A$ and $\overline{\fh}_A$ are trivial, we have
\[ \eta\cdot\ff(\sigma)(\eta)^{-1}\in A\mbox{ and }\eta\cdot \fh(\sigma)(\eta)^{-1}\in A.\]
Since $\fg(A) = A$, there exist $\sigma_{\eta,\ff},\sigma_{\eta,\fh}\in A$ such that
\[\fg(\sigma_{\eta,\ff}) =  \eta\cdot\ff(\sigma)(\eta)^{-1}\mbox{ and }\fg(\sigma_{\eta,\fh}) = \eta\cdot \fh(\sigma)(\eta)^{-1}.\]
Let us rewrite the above as
\begin{align*}
\fg(\sigma_{\eta,\ff})\fg(\sigma)^{-1} & = \eta\fg(\sigma)^{-1}\fh(\sigma)(\eta)^{-1},\\
\fg(\sigma_{\eta,\fh})\fg(\sigma) & = \eta\fg(\sigma)\ff(\sigma)(\eta)^{-1}.\end{align*}
We may now prove the next lemmas.

\begin{lem}\label{fhA lem}The following are true.
\begin{enumerate}[(a)]
\item If $\ker(\ff)=A$, then $\GG$ is normalized by $\rho(N)$.
\item If $\ker(\fh)=A$, then $\GG$ is normalized by $\lambda(N)$.
\end{enumerate}
\end{lem}
\begin{proof}Suppose that $\ker(\ff)=A$. For any $\sigma\in G$ and $\eta\in N$, we have
\[ \rho(\eta)\cdot \rho(\fg(\sigma))\ff(\sigma)\cdot \rho(\eta)^{-1} = \rho(\fg(\sigma_{\eta,\fh})\fg(\sigma))\cdot\ff(\sigma),\]
where $\sigma_{\eta,\fh}\in A$. Since $\ker(\ff)=A$, from Proposition~\ref{h prop}(c), we deduce that
\[\rho(\fg(\sigma_{\eta,\fh})\fg(\sigma))\cdot\ff(\sigma) = \rho(\fg(\sigma_{\eta,\fh}\sigma))\cdot \ff(\sigma_{\eta,\fh}\sigma),\]
whence $\GG$ is normalized by $\rho(N)$. This proves (a). A similar argument using (\ref{GG2}) and Proposition~\ref{h prop}(d) shows (b).
\end{proof}

\begin{lem}The kernels $\ker(\ff)$ and $\ker(\fh)$ are not both trivial or both $A$.
\end{lem}
\begin{proof}Recall from Proposition~\ref{h prop}(b) that $\fg^{-1}(Z(N))$, which has size $|Z(N)|$ because $\fg$ is bijective, is precisely the set of fixed points of $(\ff,\fh)$. We have
\[ |Z(N)| = \begin{cases}
p &\mbox{if $N\simeq A\times C_p$},\\ 1 &\mbox{if $N$ is almost simple with socle $A$},
\end{cases}\]
where the latter holds by Lemma~\ref{AS lem}(a). Then, clearly $\ker(\ff)$ and $\ker(\fh)$ are not both $A$, because elements of $\ker(\ff)\cap\ker(\fh)$ are fixed points of $(\ff,\fh)$.

\vspace{1mm}

Suppose for contradiction that both $\ff$ and $\fh$ are injective. If $N\simeq A\times C_p$, then in the notation of Lemma~\ref{N lem}(c), both $\ff(A),\fh(A)\simeq A$ project trivially onto $\Aut(C_p)\simeq C_{p-1}$, whence they lie in $\Aut(A)$. If $N$ is almost simple with socle $A$, then $\Aut(N)$ embeds into $\Aut(A)$ by Lemma~\ref{AS lem}(b). In both cases, we deduce from Lemma~\ref{CFSG lem}(c) that $\ff(A)=\fh(A)$, which we shall denote by $\mathfrak{A}$. Then, via restriction $\ff$ and $\fh$ induce isomorphisms
\[ \res(\ff),\res(\fh): A\longrightarrow \mathfrak{A},\mbox{ and } \res(\ff)^{-1}\circ \res(\fh) \in \Aut(A).\]
The set of fixed points of $\res(\ff)^{-1}\circ\res(\fh)$ is equal to $\fg^{-1}(Z(N))\cap A$, which is trivial because $\fg(A)=A$. This contradicts Lemma~\ref{CFSG lem}(b).
\end{proof}

\begin{proof}[Proof of Proposition~$\ref{1st prop}$]
To summarize, we have shown:
\begin{enumerate}[$\bullet$]
\item If $\ker(\fh)=1$ and $\ker(\ff)=A$, then $\GG$ is normalized by $\rho(N)$ but not $\lambda(N)$.
\item If $\ker(\ff)=1$ and $\ker(\fh)=A$, then $\GG$ is normalized by $\lambda(N)$ but not $\rho(N)$.
\end{enumerate}
Moreover, these are the only possibilities, and so the claim follows.
\end{proof}

\subsection{An alternative formula} Let us now prove:

\begin{prop}\label{formula prop}We have
 \[ e(G,N) = 2\cdot\#\left\lbrace\begin{array}{c}\mbox{regular subgroups of $\Hol(G)$ other than $\lambda(G)$}\\\mbox{which are isomorphic to $N$ and normalized by $\lambda(G)$}\end{array}\right\rbrace.\]
 \end{prop}
 
We shall prove this using (\ref{e count}) directly. Given a subgroup $\mathcal{N}$ of $\Perm(G)$, denote by $\mathcal{N}^\star$ its centralizer in $\Perm(G)$. In the case that $\mathcal{N}$ is regular:
\begin{enumerate}[$\bullet$]
\item $\mathcal{N}^\star\simeq \mathcal{N}$ and $(\mathcal{N}^\star)^\star = \mathcal{N}$;
\item $\mathcal{N}^\star$ is also regular;
\item $\mathcal{N} = \mathcal{N}^\star$ if and only if $\mathcal{N}$ is abelian;
\item $\mathcal{N}$ is normalized by $\lambda(G)$ if and only $\mathcal{N}^\star$ is normalized by $\lambda(G)$.
\end{enumerate}
These facts are all easy to prove; see \cite[Lemmas 2.1 and 2.3]{Truman}, for example. Since $N$ is non-abelian, we see that the regular subgroups of $\Perm(G)$ which are isomorphic to $N$ and normalized by $\lambda(G)$ come in pairs.

\begin{lem}\label{reverse lem} Let $\mathcal{N}$ be any regular subgroup of $\Perm(G)$ which is isomorphic to $N$ and normalized by $\lambda(G)$. If $\mathcal{N}$ is not equal to $\lambda(G)$ or $\rho(G)$, then exactly one of $\mathcal{N}$ and $\mathcal{N}^\star$ lies in $\Hol(G)$.
\end{lem}
\begin{proof}The bijection $\xi_{\mathcal{N}}$ as in the introduction induces an isomorphism
\[ \Xi_{\mathcal{N}}: \Perm(\mathcal{N})\longrightarrow \Perm(G);\hspace{1em}\Xi_{\mathcal{N}}(\pi) = \xi_{\mathcal{N}}\circ\pi\circ\xi_{\mathcal{N}}^{-1}\]
under which $\lambda(\mathcal{N})$ is sent to $\mathcal{N}$. Notice that $\rho(\mathcal{N})$ is the centralizer of $\lambda(\mathcal{N})$ in $\Perm(\mathcal{N})$ and so is sent to $\mathcal{N}^\star$. Let $\GG$ denote the preimage of $\lambda(G)$ under $\Xi_\mathcal{N}$, which is a regular subgroup of $\Perm(\mathcal{N})$ isomorphic to $G$. In summary:
\[\Xi_\mathcal{N}:\hspace{1em} \lambda(\mathcal{N}) \mapsto \mathcal{N},\,\ \rho(\mathcal{N})\mapsto \mathcal{N}^\star,\,\ \GG\mapsto \lambda(G).\]
Recall that $\Hol(\mathcal{N})$ is the normalizer of $\lambda(\mathcal{N})$ in $\Perm(\mathcal{N})$. Since $\lambda(G)$ normalizes $\mathcal{N}$, we see that $\GG$ lies in $\Hol(\mathcal{N})$. Similarly, we have
\begin{align*}
 \mathcal{N}\mbox{ normalizes }\lambda(G) &\iff \lambda(\mathcal{N}) \mbox{ normalizes }\GG,\\
  \mathcal{N}^\star\mbox{ normalizes }\lambda(G) &\iff \rho(\mathcal{N}) \mbox{ normalizes }\GG.
\end{align*}
If $\mathcal{N}$ is not equal to $\lambda(G)$ or $\rho(G)$, then $\mathcal{G}$ is not equal to $\lambda(\mathcal{N})$ or $\rho(\mathcal{N})$, and the above together with Proposition~\ref{1st prop} show that exactly one of $\mathcal{N}$ and $\mathcal{N}^\star$ normalizes $\lambda(G)$. The claim now follows.
\end{proof}

\begin{proof}[Proof of Proposition~$\ref{formula prop}$] Define
\begin{align*}
 \kappa(N) & = \#\left(\{\lambda(G),\rho(G)\}\cap\{\mbox{groups isomorphic to $N$\}}\right) \\
 & = \begin{cases} 2 & \mbox{if $N\simeq G$},\\ 0 &\mbox{if $N\not\simeq G$}.\end{cases}\end{align*}
By Lemma~\ref{reverse lem}, the number $e(G,N)$ in (\ref{e count}) is equal to 
 \[ \kappa(N) +  2\cdot\#\left\lbrace\begin{array}{c}\mbox{regular subgroups of $\Hol(G)$ other than $\lambda(G),\rho(G)$}\\\mbox{which are isomorphic to $N$ and normalized by $\lambda(G)$}\end{array}\right\rbrace .\]
The claim is then clear.
\end{proof}

\section{The case when $N= A\times C_p$}\label{direct prod sec}

In this section, assume that $N=A\times C_p$, and fix a generator $\ep$ of $C_p$. We shall apply Proposition~\ref{formula prop} to prove Theorem~\ref{thm1}. Let us define 
\[\mathrm{InHol}(G) = \rho(G)\rtimes\Inn(G)\]
to be the \emph{inner holomorph} of $G$, which is a subgroup of $\Hol(G)$.

\begin{lem}\label{InHol lem} A regular subgroup of $\Hol(G)$ isomorphic to $N$ lies in $\mathrm{InHol}(G)$.
\end{lem}
\begin{proof}Let $\mathcal{N}$ be a regular subgroup of $\Hol(G)$ isomorphic to $N$. Write
\[ \mathcal{N} = \{\rho(\fg(\eta))\cdot\ff(\eta):\eta\in N\},\mbox{ where }\begin{cases}
\ff\in\Hom(N,\Aut(G))\\\fg\in Z_\ff^1(N,G)\mbox{ is bijective}
\end{cases}\]
as in Proposition~\ref{fg prop}, and let $\fh\in\Hom(N,\Aut(G))$ be as in (\ref{h def}). We have
\[ \mathcal{N} \subset \mathrm{InHol}(G) \iff \ff(N) \subset \Inn(G) \iff \fh(N)\subset \Inn(G).\]
Since $G$ has trivial center by Lemma~\ref{AS lem}(a), the pair $(\ff,\fh)$ is fixed point free by Proposition~\ref{h prop}(b). It follows that $A$ cannot lie in both $\ker(\ff)$ and $\ker(\fh)$, so at least one of $\ff$ and $\fh$ is injective on $A$.

\vspace{1mm}

Without loss of generality, let us assume that $\ff$ is injective on $A$. Then, by Lemmas~\ref{AS lem}(b) and~\ref{CFSG lem}(c), we see that $\ff(A)\simeq A $ is the subgroup of $\Inn(G)$ consisting of the inner automorphisms
\[ \conj(\sigma)\in\Inn(G);\hspace{1em}\conj(\sigma)(x) = \sigma x\sigma^{-1}\hspace{1em}\mbox{ for }\sigma\in A.\]
Put $\theta = \ff(\ep)$, which commutes with $\ff(A)$. But then $\sigma^{-1}\theta(\sigma)$ lies in the center of $G$ for all $\sigma\in A$ because for any $x\in G$, we have
\[ \sigma\theta(x)\sigma^{-1} = (\conj(\sigma)\circ\theta)(x) =
(\theta\circ\conj(\sigma))(x) = \theta(\sigma)\theta(x)\theta(\sigma)^{-1}.\]
Since $G$ has trivial center, we deduce that $\theta|_A = \mathrm{Id}_A$, and so in fact $\theta=\mbox{Id}_G$ by Lemma~\ref{AS lem}(b). This proves $\ff(N) = \ff(A)$, whence the claim. 
\end{proof}

Now, since $G$ has trivial center, the regular subgroups of $\mathrm{InHol}(G)$ isomorphic to $N$ are precisely the subgroups of the shape
\[ \mathcal{N}_{(f,h)} = \{\rho(h(\eta))\cdot\lambda(f(\eta)):\eta\in N\}\]
as $f,h$ range over $\Hom(N,G)$ with $(f,h)$ fixed point free, by \cite[Proposition 6]{Byott Childs} or \cite[Subsection 2.3.1]{Tsang HG}. Moreover, each $\mathcal{N}$ correspond to exactly $|\Aut(N)|$ pairs of $(f,h)$. By Proposition~\ref{formula prop} and Lemma~\ref{InHol lem}, we then see that
\[ e(G,N) = 2\cdot\frac{1}{|\Aut(N)|}\cdot\#\left\lbrace\begin{array}{c}\mbox{fixed point free $(f,h)$ for $f,h\in\Hom(N,G)$}\\\mbox{such that $\mathcal{N}_{(f,h)}$ is normalized by $\lambda(G)$}\end{array}\right\rbrace.\]
In what follows, let $f,h\in\Hom(N,G)$. Note that both $\ker(f)$ and $\ker(h)$ are non-trivial because $N$ is not isomorphic to $G$. For the pair $(f,h)$ to be fixed point free, the subgroups $\ker(f)$ and $\ker(h)$ must intersect trivially, whence by Lemma~\ref{N lem}(a), exactly one of them is $A$ and the other is $C_p$. Also, notice that by Lemma~\ref{CFSG lem}(c), we must have $h(A)=A$ if $\ker(h)=C_p$ and similarly $f(A)=A$ if $\ker(f)=C_p$.

\begin{lem}\label{normalize lem}Let $\mathcal{N} = \mathcal{N}_{(f,h)}$ be as above.
\begin{enumerate}[(a)]
\item If $\ker(h) = C_p$ and $\ker(f) = A$, then $\mathcal{N}$ is not normalized by $\lambda(G)$.
\item If $\ker(f) = C_p$ and $\ker(h) = A$, then $\mathcal{N}$ is normalized by $\lambda(G)$.
\end{enumerate}
\end{lem}
\begin{proof}For any $\eta\in N$ and $\sigma\in G$, we have
\[\lambda(\sigma)\cdot \rho(h(\eta))\lambda(f(\eta))\cdot\lambda(\sigma)^{-1} = \rho(h(\eta))\cdot \lambda(\sigma f(\eta)\sigma^{-1}).\]
Note that $\rho(G)$ and $\lambda(G)$ intersect trivially since $G$ has trivial center. Thus, for $\mathcal{N}$ to be normalized by $\lambda(G)$, the subgroup $f(N)$, which is non-trivial in both parts, is normal in $G$ and in particular contains $A$. This yields (a).

\vspace{1mm}

Now, suppose that $\ker(f)=C_p$ and $\ker(h)=A$. Write $\eta = a\ep^i$ for $a\in A$ and $i\in\mathbb{Z}$. Since $f(A) = A$ and $A$ is normal in $G$, there exists $a_\sigma\in A$ such that $f(a_\sigma) = \sigma f(a)\sigma^{-1}$. It follows that
\[ \rho(h(\eta))\cdot \lambda(\sigma f(\eta)\sigma^{-1}) = \rho(h(\ep^i))\cdot\lambda(\sigma f(a)\sigma^{-1}) = \rho(h(a_\sigma\ep^i))\cdot\lambda(f(a_\sigma\ep^i)),\]
which lies in $\mathcal{N}$. This proves (b).
\end{proof}

\begin{lem}\label{fpf lem}Suppose that $\ker(f)=C_p$ and $\ker(h)=A$. Then $(f,h)$ is fixed point free if and only if $h(\ep)\not\in A$.
\end{lem}
\begin{proof}Again $f(A)=A$. If $h(\ep)\in A$, then $f(a) = h(\ep)$ for $a\in A$, and $a\ep\neq1_N$ is a fixed point of $(f,h)$. If $h(\ep)\notin A$, then $f(N)\cap h(N)$ is trivial, and $(f,h)$ is fixed point free because $\ker(f)\cap \ker(h)$ is also trivial.
\end{proof}

\begin{proof}[Proof of Theorem~$\ref{thm1}$] 
By Lemmas~\ref{normalize lem} and~\ref{fpf lem} we have
 \[ e(G,N) = 2\cdot\frac{1}{|\Aut(N)|}\cdot e_1(G,N)\cdot e_2(G,N),\]
where we define
\begin{align*}
e_1(G,N) & = \#\{f\in\Hom(N,G):\ker(f)=C_p\},\\
e_2(G,N) & = \#\{h\in\Hom(N,G):\ker(h)=A,\, h(\ep)\notin A\}.
\end{align*}
We have $|\Aut(N)|=(p-1)|\Aut(A)|$ by Lemma~\ref{N lem}(c). Also, it is clear that
\[e_2(G,N) = \#\{\sigma\in G\setminus A:\sigma\mbox{ has order $p$}\},\mbox{ and }e_1(G,N) = |\Aut(A)|\]
because $f(A)=A$ whenever $\ker(f) = C_p$. The theorem now follows.\end{proof}

\section{The case when $N$ is non-perfect}\label{case2 sec}

In this section, assume that $N$ is non-perfect and $e(G,N)$ is non-zero. We shall prove Theorem~\ref{thm2}. By (\ref{B formula}) and Proposition~\ref{fg prop}, there exist
\[ \ff\in \Hom(G,\Aut(N))\mbox{ and a bijective }\fg\in Z_\ff^1(G,N).\]
Since $N$ is non-perfect, it has a maximal characteristic subgroup $M$ containing $[N,N]$. We shall show that $M\simeq A$.

\vspace{1mm}

Since $M$ contains $[N,N]$, from (\ref{N/M}), we see that
\[ N/M\simeq (\bZ/\ell\bZ)^m,\mbox{ where $\ell$ is prime and $m\in\mathbb{N}$}.\]
Recall that $\ff$ and $\fg$, respectively, induce
\[ \overline{\ff}_M\in\Hom(G,\Aut(N/M))\mbox{ and a surjective }\overline{\fg}_M\in Z_{\overline{\ff}_M}^1(G,N/M)\]
as in Proposition~\ref{char prop}. Put $H = \fg^{-1}(M)$, which is a subgroup of $G$ by Proposition~\ref{char prop}(b), and has order $|M|$ because $\fg$ is bijective. Note that
\begin{equation}\label{A index} [A:H\cap A]  = [AH:H] = \ell^m/[G:AH],\mbox{ and }[G:AH]=1\mbox{ or }p.\end{equation}
In the case $[G:AH]=1$, we shall use the next lemma.

\begin{lem}\label{not perfect lem1}If $A$ has a subgroup of index $\ell^m$, then $A\simeq \mathrm{PSL}_2(7)$, or $A$ does not embed into $\mathrm{GL}_m(\ell)$.
\end{lem}
\begin{proof}See \cite[Lemmas 4.2 and 4.4]{Byott simple}.
\end{proof}

In the case $[G:AH]=p$, note that $\ell = p$ necessarily, and we shall use the next two lemmas. Their proofs are refinements of \cite[Section 4]{Byott simple}. A key fact is \cite[Theorem 1]{RG}, which gives the subgroups of prime power index in $A$, and its proof uses the classification of finite simple groups. We shall also use the hypothesis that $A$ is a subgroup of index $p$ in $G$, which means that $p$ divides the order of the outer automorphism group $\Out(A)$ of $A$.

\begin{lem}\label{not perfect lem2} If $A$ has a subgroup of index $p^{m-1}$ with $m\geq2$, then 
\begin{equation}\label{A PSL} A\simeq \mathrm{PSL}_n(q)\mbox{ with }p^{m-1} = \frac{q^n-1}{q-1},\end{equation}
or $G$ does not embed into $\mathrm{GL}_m(p)$.
\end{lem}
\begin{proof}Suppose that $A$ has a subgroup of index $p^{m-1}$ with $m\geq2$. Then, by \cite[Theorem 1]{RG}, one of the following holds.
\begin{enumerate}[(a)]
\item $A\simeq A_{p^{m-1}}$ with $p^{m-1}\geq 5$;
\item $A\simeq\mathrm{PSL}_n(q)$ with $p^{m-1} = (q^n-1)/(q-1)$;
\item $A\simeq\mathrm{PSL}_2(11)$ with $p^{m-1}=11$;
\item $A\simeq M_{11}$ with $p^{m-1}=11$, or $A\simeq M_{23}$ with $p^{m-1}=23$;
\item $A\simeq\mathrm{PSU}_4(2)$ with $p^{m-1}=27$.
\end{enumerate}
Recall that $p$ divides the order of $\Out(A)$. Since
\[ |\Out(\mathrm{PSL}_2(11))| = 2 = |\Out(\mathrm{PSU}_4(2))|,\,\ |\Out(M_{11})| = 1 = |\Out(M_{23})|,\]
cases (c),(d),(e) do not occur. Since $|\Out(A_n)|=2$ for all $n\geq 5$ with $n\neq 6$, we must  have $p=2$ with $m\geq 4$ and $G \simeq S_{2^{m-1}}$ in case (a). Notice that $S_{2^{m-1}}$ does not embed into $\mathrm{GL}_m(2)$ for $m\geq 4$ because
\begin{align*} |\mathrm{GL}_m(2)| &= 2^{m(m-1)/2}\cdot s\mbox{ with $s\in\mathbb{N}$ odd},\\
 |S_{2^{m-1}}| &= 2\cdot 2^2\cdots 2^{m-1}\cdot 6\cdot t = 2^{m(m-1)/2+1}\cdot 3t\mbox{ with }t\in\mathbb{N}. \end{align*}
We are left with case (b) and the claim now follows.
\end{proof}

To deal with the remaining case in (\ref{A PSL}), we shall follow \cite[Section 4]{Byott simple} and use \cite{PSL1,PSL2}, which give lower bounds for the degrees of projective irreducible representations of projective special linear groups in cross characteristics. In particular, we shall use the version stated in \cite[Theorem 4.3]{Byott simple}.

\begin{lem}\label{not perfect lem3}If $A$ is as in $(\ref{A PSL})$ with $m\geq 2$, then $A\simeq\mathrm{PSL}_2(7)$, or $A$ does not embed into $\mathrm{GL}_m(p)$.
\end{lem}
\begin{proof}Suppose that $A$ is as in (\ref{A PSL}) with $m\geq 2$, and in particular
\begin{equation}\label{pm} p^{m-1} = \frac{q^n-1}{q-1}.\end{equation}
We already know by \cite[Lemma 4.1(a)]{Tsang solvable} that $A$ does not embed into $\mathrm{GL}_2(p)$. Hence, we may assume $m\geq 3$, and together with (\ref{pm}), we deduce that
\[ (n,q) \neq (3,2),(2,4),(3,4), (4,2), (4,3), (2,9).\]
Suppose now that $A$ embeds into $\mathrm{GL}_m(p)$. By \cite[Theorem 4.3]{Byott simple}, we have:
\begin{enumerate}[$\bullet$] 
\item If $n\geq 3$, then $m\geq (q^n-q)/(q-1)-1$;
\item If $n=2$, then $m\geq (q-1)/\gcd(q-1,2)$.
\end{enumerate}
 In the first case, we have
\[ m \geq \frac{q^n-q}{q-1}-1 = \frac{q^n-1}{q-1} - 2 = p^{m-1}-2.\]
Since $m\geq 3$, this yields $(p,m)=(2,3)$, which cannot satisfy (\ref{pm}) for $n\geq 3$. In the second case, we have
\[ m \geq \frac{q-1}{\gcd(q-1,2)} = \frac{p^{m-1}-2}{\gcd(p^{m-1}-2,2)}\geq \frac{p^{m-1}-2}{2}.\]
Since $m\geq 3$, this yields $(p,m)=(2,3)$ or $(2,4)$, which corresponds to $q=3$ or $7$, respectively, for $n=2$. But $\mathrm{PSL}_2(3)$ is non-simple, so we are left with the case $A\simeq \mathrm{PSL}_2(7)$, whence the claim.
\end{proof}

\begin{lem}\label{H=A}If $A\not\simeq\mathrm{PSL}_2(7)$, then $\fg^{-1}(M)=A$ and $[N:M] = p$.
\end{lem}
\begin{proof}We have $H = \fg^{-1}(M)$ by definition and recall the equalities in (\ref{A index}). There are three cases, and recall that $\ell=p$ necessarily when $[G:AH]=p$. 
\begin{enumerate}[$(1)$]
\item $[G:AH]=1$;
\item $[G:AH]=p$ and $m=1$;
\item $[G:AH]=p$ and $m\geq 2$.
\end{enumerate}
Let us first prove that $A\subset H$. In case (2), we have $[A:H\cap A]=1$, and so clearly $A\subset H$. In cases (1) and (3), suppose that $A\not\simeq\mathrm{PSL}_2(7)$. Then, since the range of $\overline{\ff}_M$ is equal to
\[ \Aut(N/M)\simeq \Aut((\bZ/\ell\bZ)^m)\simeq \mathrm{GL}_m(\ell),\]
we deduce from Lemma~\ref{not perfect lem1},~\ref{not perfect lem2}, and~\ref{not perfect lem3} that $\overline{\ff}_M$ is not injective. From (\ref{normal sub}), it follows that $\ker(\overline{\ff}_M)$ has to contain $A$, whence $(\overline{\fg}_M)|_A$ is a homomorphism by Proposition~\ref{h prop}(c). Since the range of $\overline{\fg}_M$ is equal to
\[ N/M \simeq (\bZ/\ell\bZ)^m,\]
necessarily $(\overline{\fg}_M)|_A$ is trivial, which means that $A\subset H$. In all three cases, we have $A\subset H$. Since $A$ has index $p$ in $G$ and $H\subsetneq G$, we must have $H=A$. This in turn implies $[N:M] = [G:A] = p$, as claimed.
\end{proof}

\begin{proof}[Proof of Theorem~$\ref{thm2}$] Suppose first that $A\simeq \mathrm{PSL}_2(7)$. Then $G\simeq \mathrm{PGL}_2(7)$, and by \cite[Theorem 1.10]{Tsang solvable}, we know that
\[ e(\mathrm{PGL}_2(7),N) =0\mbox{ for all solvable $N$}.\]
Since $\mathrm{PGL}_2(7)$ and $\mathrm{PSL}_2(7)\times C_2$ are the only non-perfect insolvable groups of order $336$, we see that Theorem~\ref{thm2} holds in this case.

\vspace{1mm}

Suppose now that $A\not\simeq\mathrm{PSL}_2(7)$. Then $\fg^{-1}(M) = A$ by Lemma~\ref{H=A}, and so $e(A,M)\neq0$ by Proposition \ref{char prop}(c). Since $A$ is non-abelian simple, by \cite{Byott simple}, this implies $M\simeq A$. Since $[N:M]=p$, the theorem follows from Lemma~\ref{copy of A lem}.
\end{proof}

\section{The case when $N$ is perfect}\label{case3 sec}

In this section, assume that $N$ is perfect and $e(G,N)$ is non-zero. We shall \par\noindent prove Theorem~\ref{thm3}. As in Section~\ref{case2 sec}, by (\ref{B formula}) and Proposition~\ref{fg prop}, there exist
\[ \ff\in\Hom(G,\Aut(N))\mbox{ and a bijective }\fg\in Z_\ff^1(G,N).\]
Also, let $\fh\in\Hom(G,\Aut(N))$ be as in (\ref{h def}). Let $M$ be any maximal characteristic subgroup of $N$. We shall show that $M=Z(N)$ and $N/M\simeq A$.

\vspace{1mm}

Since $N$ is perfect, from (\ref{N/M}), we see that
\[ N/M \simeq T^m,\mbox{ where $T$ is non-abelian simple and $m\in\mathbb{N}$}.\]
Recall that $\ff$ and $\fg$, respectively, induce
\[ \overline{\ff}_M\in\Hom(G,\Aut(N/M))\mbox{ and a surjective }\overline{\fg}_M\in Z_{\overline{\ff}_M}^1(G,N/M)\]
as in Proposition~\ref{char prop}. 
 
\begin{lem}\label{A into T}The group $A$ embeds into $T$.
\end{lem}
\begin{proof}It is known, by \cite[Lemma 3.2]{Byott simple} for example, that
\[ \Aut(N/M) \simeq \Aut(T^m) \simeq \Aut(T)^m\rtimes S_m.\]
There exists a prime $r\neq p$ which divides $|T|$ because groups of prime power order are nilpotent. Then, since
\[ p |A| = |G| = |N| = |M||T|^m,\mbox{ we have }r^m \mbox{ divides }|A|.\]
But $r^m$ does not divide $m!$ as in the proof of \cite[Lemma 3.3]{Byott simple}. It follows that $A$ cannot embed into $S_m$ and so the homomorphism
\[\begin{tikzcd}[column sep = 2cm]
A \arrow{r}{\overline{\ff}_M} & \Aut(N/M) \arrow[equal]{r}{\mbox{\tiny identification}} & \Aut(T)^m\rtimes S_m \arrow{r}{\mbox{\tiny projection}}& S_m
\end{tikzcd}\]
is trivial. Since $\Out(T)$ is solvable by Lemma~\ref{CFSG lem}(a), the homomorphism
\[\begin{tikzcd}[column sep = 2cm]
A \arrow{r}{\overline{\ff}_M} & \Aut(T)^m \arrow{r}{\mbox{\tiny quotient}}& \Out(T)^m
\end{tikzcd}\]
is also trivial. We then see that $\overline{\ff}_M(A)$ lies in $\Inn(T)^m$.
\begin{enumerate}[$\bullet$]
\item If $(\overline{\ff}_M)|_A$ is injective, then clearly $A$ embeds into $\Inn(T)^m\simeq T^m$.
\item If $(\overline{\ff}_M)|_A$ is trivial, then $(\overline{\fg}_M)|_A$ is a homomorphism by Proposition~\ref{h prop}(c). But $(\overline{\fg}_M)|_A$ cannot be trivial, for otherwise $A\subset\fg^{-1}(M)$, and\\
\[ p = |G|/|A| \geq |G|/|\fg^{-1}(M)| = |N|/|M| = |T|^m,\]
which is impossible. It follows that $(\overline{\fg}_M)|_A$ must be injective, so $A$ embeds into $N/M\simeq T^m$.
\end{enumerate}
In both cases $A$ embeds into $T^m$. Observe that the projection of $A$ onto the $m$ components of $T^m$ cannot be all trivial, so in fact $A$ embeds into $T$. 
\end{proof}

As in Section~\ref{case2 sec}, we shall use \cite[Theorem 1]{RG} and also the hypothesis that $A$ has index $p$ in $G$. The former lists the subgroups of prime power index in a finite non-abelian simple group while the latter implies that $p$ divides the order of the outer automorphism group $\Out(A)$ of $A$.

\begin{lem}\label{m=1}We have $m=1$ and $|M| = p$.
\end{lem}
\begin{proof}By Lemma~\ref{A into T}, we know that $A$ embeds into $T$, and write $|T| = d|A|$ for $d\in\mathbb{N}$. Then, we have
\[ p|A| = |G| = |N| = |M||T|^m = d^m|A|^m|M|,\mbox{ and so }p = d^m|A|^{m-1}|M|.\]
This gives $m=1$, and $|M|=1$ or $p$. Suppose for contradiction that $|M|=1$, in which case $N\simeq T$ and $A$ embeds into $T$ as a subgroup of index $p$. Since $T$ is non-abelian simple, one of the following holds by \cite[Theorem 1]{RG}.
\begin{enumerate}[(a)]
\item $T\simeq A_p$ and $A\simeq A_{p-1}$ with $p\geq 5$;
\item $T\simeq \mathrm{PSL}_n(q)$ with $p=(q^n-1)/(q-1)$;
\item $T\simeq \mathrm{PSL}_2(11)$ and $A\simeq A_5$ with $p=11$;
\item $T\simeq M_{11}$ and $A\simeq M_{10}$ with $p=11$, or
 \\$T\simeq M_{23}$ and $A\simeq \mathrm{M}_{22}$ with $p=23$.
\end{enumerate}
Note that $\mathrm{M}_{10}$ is non-simple. Since $p$ divides $|\Out(A)|$ while
\[ |\Out(A_n)| = 2\mbox{ or }4\mbox{ for $n\geq 5$}\mbox{ and } |\Out(M_{22})|=2,\]
cases (a), (c), and (d) do not occur. To deal with case (b), note that $N\simeq T$ has trivial center, so $(\ff,\fh)$ is fixed point free by Proposition~\ref{h prop}(b). Thus, the intersection $\ker(\ff)\cap\ker(\fh)$ is trivial, and by (\ref{normal sub}), at least one of $\ff$ and $\fh$ has to be injective. Since $N$ is not isomorphic to $G$, and by definition
\[ \ff(G) \subset \Inn(N) \iff \fh(G)\subset \Inn(N),\]
the image $\ff(G)$ cannot lie in $\Inn(N)\simeq N$. It follows 
the homomorphism
\[\begin{tikzcd}[column sep = 2cm]
G \arrow{r}{\ff} & \Aut(N) \arrow{r}{\mbox{\tiny quotient}} & \Out(N) \arrow{r}{\simeq} & \Out(T)
\end{tikzcd}\]
is non-trivial. From (\ref{normal sub}), we then deduce that $p$ has to divide $|\Out(T)|$. But for $n\geq 2$, by \cite[Theorem 3.2]{Wilson} for example, we know that
\[ |\Out(\mathrm{PSL}_n(q))| = 2\gcd(n,q-1)f \mbox{ or }\gcd(n,q-1)f,\]
where $q =r^f$ with $r$ a prime. In case (b), note that
\[ p = (q^n-1)/(q-1) = q^{n-1}+\cdots +q+1\geq q+1 > \max\{2, q-1, f\},\]
and we see that $p$ cannot divide $|\Out(\mathrm{PSL}_n(q))|$. Hence, all four cases (a) to (d) are impossible, so necessarily $|M|=p$, as desired.
\end{proof}

\begin{proof}[Proof of Theorem~$\ref{thm3}$ Condition (1)] So far, we have shown that 
\[A\mbox{ embeds into $T$},\, N/M\simeq T,\mbox{ and }|M|=p.\]
By comparing orders, in fact $T\simeq A$. We have a homomorphism
\[ N\longrightarrow \Aut(M);\hspace{1em}\eta\mapsto (x\mapsto \eta x\eta^{-1})\]
because $M$ is normal. But it must be trivial since $N$ is perfect while $\Aut(M)$ is cyclic. This means that $M\subset Z(N)$, and so $M=Z(N)$ by the maximality of $M$. This claim then follows.
\end{proof}

Now, we know that $N$ is quasisimple, with $N/Z(N)\simeq A$ and $|Z(N)|=p$. Using this, we may prove the next two lemmas.

\begin{lem}\label{A in N lem}There is no subgroup isomorphic to $A$ in $N$.
\end{lem}
\begin{proof}Suppose for contradiction that $B$ is such a subgroup. Then, we have
\[ |BZ(N)| = |B||Z(N)|/|B\cap Z(N)| = p|A| = |G| = |N|,\]
where $B\cap Z(N)$ is trivial because $B$ has trivial center. But this implies that 
\[N = BZ(N)\mbox{ and in particular $[N,N]=[B,B]$}.\]
This is impossible because $B\subsetneq N$ and $N$ is perfect.
\end{proof}

\begin{lem}\label{A into Inn}Both $\ff$ and $\fh$ embed $A$ into $\Inn(N)$.
\end{lem}
\begin{proof}Notice that $\Out(N)$ is solvable by Lemma~\ref{CFSG lem}(a); see the proof of \cite[Lemma 3.6]{Tsang HG}. Since $A$ is perfect, the homomorphisms
\[\begin{tikzcd}[column sep = 2cm]
A \arrow{r}{\ff,\fh} & \Aut(N) \arrow{r}{\mbox{\tiny quotient}} & \Out(N)
\end{tikzcd}\]
are trivial, whence $\ff(A)$ and $\fh(A)$ lie in $\Inn(N)$. Observe that the map
\[A\longrightarrow N;\hspace{1em}\begin{cases}
x\mapsto\fg(x) & \mbox{if $\ff|_A$\mbox{ were trivial}}\\ 
x\mapsto\fg(x)^{-1}&\mbox{if $\fh|_A$ were trivial}
\end{cases}\]
would be a homomorphism by Propositions~\ref{h prop}(c),(d), and so $A$ would embed into $N$ because $\fg$ is bijective. But this is impossible by Lemma~\ref{A in N lem}, so both $\ff$ and $\fh$ are injective on $A$, as desired.
\end{proof}

Lemma~\ref{A into Inn} tells us that $\ff$ and $\fh$, respectively, induce isomorphisms
\[ f,h : A\longrightarrow N/Z(N);\hspace{1em}\begin{cases}f(\sigma) = \widetilde{f}(\sigma)Z(N),\\h(\sigma) = \widetilde{h}(\sigma)Z(N),\end{cases}\]
where $\widetilde{f}(\sigma),\widetilde{h}(\sigma)\in N$ are such that for all $x\in N$, we have
\[ \ff(\sigma)(x) = \widetilde{f}(\sigma)x\widetilde{f}(\sigma)^{-1}\mbox{ and }\fh(\sigma)(x)= \widetilde{h}(\sigma)x\widetilde{h}(\sigma)^{-1}.\]
Since $\fg$ is bijective, by Proposition~\ref{char prop}(b), we know that $\fg^{-1}(Z(N)) = \langle\zeta\rangle$ for some element $\zeta\in G$ of order $p$. Note also that $Z(N) = \langle\fg(\zeta)\rangle$.

\begin{proof}[Proof of Theorem~$\ref{thm3}$ Condition (2)] Consider $\varphi = f^{-1}\circ h$, which is an automorphism on $A$. For any $\sigma\in A$, we have
\[ \varphi(\sigma) = \sigma\iff f(\sigma) = h(\sigma) \iff \ff(\sigma) = \fh(\sigma) \iff \sigma \in \langle\zeta\rangle\cap A\]
by Proposition~\ref{h prop}(b). Since $\varphi$ has a non-trivial fixed point by Lemma~\ref{CFSG lem}(b), we deduce that $\zeta\in A$, and $\varphi$ has exactly $p$ fixed points, namely the elements of $\langle\zeta\rangle$. This proves the claim. \end{proof}

Now, we also know that $\zeta\in A$, so the element $\widetilde{f}(\zeta)\in N$ is defined.

\begin{proof}[Proof of Theorem~$\ref{thm3}$ Condition (3)] Take $\widetilde{\zeta} = \widetilde{f}(\zeta)$, and $\widetilde{\zeta}Z(N) = f(\zeta)$ has order $p$ because $f$ is an isomorphism. Suppose for contradiction that there is an element $\eta\in N$ such that
\[ \eta\widetilde{f}(\zeta) \equiv \widetilde{f}(\zeta)\eta\hspace{-2mm}\pmod{Z(N)}\mbox{ but } \eta\widetilde{f}(\zeta)\neq\widetilde{f}(\zeta)\eta.\]
Since $Z(N)=\langle\fg(\zeta)\rangle$, there exists $i\in\bZ$ with $i\not\equiv0$ (mod $p$) such that
\[ \widetilde{f}(\zeta)\eta\widetilde{f}(\zeta)^{-1}\eta^{-1} = \fg(\zeta)^i,\mbox{ or equivalently }\widetilde{f}(\zeta)\eta\widetilde{f}(\zeta)^{-1} = \fg(\zeta)^i\eta.\]
Let $j\in\bZ$ be such that $ij\equiv -1$ (mod $p$), and write $\eta^j = \fg(\sigma)$, where $\sigma\in G$. Then, since $\fg(\zeta)\in Z(N)$, raising the above equation to the $j$th power yields
\[ \widetilde{f}(\zeta)\fg(\sigma)\widetilde{f}(\zeta)^{-1} = \fg(\zeta)^{-1}\fg(\sigma).\]
But this implies that
\[ \fg(\zeta\sigma) = \fg(\zeta)\cdot \ff(\zeta)(\fg(\sigma)) = \fg(\zeta)\widetilde{f}(\zeta)\fg(\sigma)\widetilde{f}(\zeta)^{-1} = \fg(\sigma),\]
which contradicts that $\fg$ is bijective. This completes the proof.
\end{proof}

\begin{lem}\label{cent lem}For any $\sigma\in G$ such that $\ff(\sigma)$ fixes $Z(N)$ pointwise, we have
\[ \sigma\zeta = \zeta\sigma \mbox{ if and only if }\fg(\sigma)\widetilde{f}(\zeta) = \widetilde{f}(\zeta)\fg(\sigma).\]
\end{lem}
\begin{proof}In the case that $\ff(\sigma)$ fixes $Z(N)$ pointwise, we have
\begin{align*}
\fg(\sigma\zeta) & = \fg(\sigma)\cdot \ff(\sigma)(\fg(\zeta)) = \fg(\sigma)\fg(\zeta),\\
\fg(\zeta\sigma) & = \fg(\zeta)\cdot \ff(\zeta)(\fg(\sigma)) = \fg(\zeta)\widetilde{f}(\zeta)\fg(\sigma)\widetilde{f}(\zeta)^{-1}.
\end{align*}
Since $\fg$ is bijective and $\fg(\zeta)\in Z(N)$, we see that the claim holds.
\end{proof}

Let us use $\Cent_*(-)$ to denote the centralizer in a given group $*$.

\begin{proof}[Proof of Theorem~$\ref{thm3}$ Condition (4)] By the proof of condition (3), the map
\[ \Cent_{N}(\widetilde{f}(\zeta)) \longrightarrow \Cent_{N/Z(N)}(f(\zeta));\hspace{1em}\eta\mapsto \eta Z(N)\]
is surjective. Its kernel is clearly $Z(N)$, and this implies that
\[ |\Cent_{N}(\widetilde{f}(\zeta))| = p\cdot |\Cent_{N/Z(N)}(f(\zeta))| =  p\cdot  |\Cent_A(\zeta)|,\]
where the second equality holds because $f$ is an isomorphism. Suppose that $Z(N)$ is fixed pointwise by $\Aut(N)$. Then, from Lemma~\ref{cent lem}, we see that
\[ |\Cent_G(\zeta)| = |\Cent_N(\widetilde{f}(\zeta))|\]
since $\fg$ is bijective. Putting the equalities together, we obtain
\[ |\Cent_{G}(\zeta)| = p\cdot |\Cent_A(\zeta)|,\]
for which the claim follows.
\end{proof}

\section{Almost simple groups of alternating or sporadic type}\label{app sec}

In this section, let $\Gamma$ be a finite almost simple group, which is non-simple, and whose socle is an alternating group or a sporadic simple group. We shall apply our theorems to determine the numbers $e(\Gamma,\Delta)$ for all groups $\Delta$ of the same order as $\Gamma$, except when $\Gamma\simeq\Aut(A_6)$.

\vspace{1mm}

First, suppose that the socle of $\Gamma$ is an alternating group. It is known that
\[ \Out(A_n) \simeq C_2\mbox{ for $n\geq 5$ with $n\neq 6$, and }\Out(A_6)=C_2\times C_2.\]
Since we assumed that $\Gamma$ is non-simple, either
\[  \Gamma\simeq S_n\mbox{ with $n\geq 5$, or }\Gamma\simeq\mathrm{PGL}_2(9),\mathrm{M}_{10},\Aut(A_6). \]
For $\Gamma\simeq S_n$ with $n\geq 5$, the numbers $e(S_n,\Delta)$ are already known by \cite{Childs simple} and \cite{Tsang Sn}. For both $\Gamma\simeq\mathrm{PGL}_2(9),\mathrm{M}_{10}$, by Theorems~\ref{thm2} and~\ref{thm3}(a), we know that
\[e(\Gamma,\Delta) \neq 0 \mbox{ only if }\Delta\simeq A_6\times C_2,S_6,\mathrm{PGL}_2(9),\mathrm{M}_{10},2 A_6,\]
where $2A_6$ is the double cover of $A_6$. By applying Theorems~\ref{thm old} and~\ref{thm1}, we computed in \textsc{Magma} \cite{magma} that

\[ \begin{cases}
e(\mathrm{PGL}_2(9),\mathrm{PGL}_2(9)) = 92\mbox{ and }e(\mathrm{PGL}_2(9),A_6\times C_2) = 72,\\
e(\mathrm{M}_{10},\mathrm{M}_{10}) = 92\mbox{ and }e(\mathrm{M}_{10},A_6\times C_2) = 0.\end{cases}\]
Since $2A_6$ does not satisfy condition (3) in Theorem~\ref{thm3}, by \cite[Lemma 2.7]{Tsang Sn} for example, we also have
\[ e(\mathrm{PGL}_2(9), 2A_6) = 0 = e(\mathrm{M}_{10},2A_6).\]
Using (\ref{B formula}) and a similar code as in the appendix of \cite{Tsang Sn}, we found that
\[\begin{cases}
e(\mathrm{PGL}_2(9),S_6) = 0\mbox{ and }e(\mathrm{PGL}_2(9),\mathrm{M}_{10})  = 60,\\
e(\mathrm{M}_{10},S_6)  = 72\mbox{ and }e(\mathrm{M}_{10},\mathrm{PGL}_2(9))  = 60.
\end{cases}\]
We have thus determined $e(\Gamma,\Delta)$ completely except when $\Gamma\simeq\Aut(A_6)$.


\begin{remark}Observe that
\[ e(\Gamma_1,\Gamma_2) = e(\Gamma_2,\Gamma_1)\mbox{ for all }\Gamma_1,\Gamma_2\in\{S_6,\mathrm{PGL}_2(9),\mathrm{M}_{10}\}\]
by the above and (\ref{S6 special}). These symmetries could possibly be a special case of a more general phenomenon, and perhaps come from the fact that
\[ \Aut(A_6)\simeq \Aut(S_6) \simeq \Aut(\mathrm{PGL}_2(9)) \simeq \Aut(\mathrm{M}_{10}),\]
together with the formulae in (\ref{B formula}) and Proposition~\ref{formula prop}.
\end{remark}

\vspace{1mm}

Next, suppose that the socle of $\Gamma$ is one of the $26$ sporadic simple groups. The outer automorphism group of a sporadic simple group $A$ has order dividing two, and is non-trivial precisely when
\[A\simeq\mathrm{M}_{12},\mathrm{M}_{22},\mbox{HS},\mathrm{J}_{2},\mbox{McL},\mbox{Suz},\mbox{He},\mbox{HN},\mbox{Fi}_{22},\mbox{Fi'}_{24},\mbox{O'N},\mathrm{J}_3,\]
where the notation is standard. Since we assumed that $\Gamma$ is non-simple, we see that $\Gamma\simeq \Aut(A)$ for one of the sporadic simple groups $A$ listed above. By Theorems~\ref{thm2} and~\ref{thm3}(a), we know that
\[ e(\Gamma,\Delta) \neq 0 \mbox{ only if }\Delta\simeq A\times C_2,\Aut(A),\mbox{ or $\Delta$ is a double cover of $A$}.\]
The element structures of $A$ as well as its covers and $\Aut(A)$ are available in the \textsc{Atlas} \cite{Atlas}. Using \cite{Atlas} and Theorem~\ref{thm old}, the number $e(\Gamma,\Gamma)$ has already been computed in \cite[p. 953]{Tsang ASG}. Similarly, we found that
\begin{longtable}{|c|c|}
\hline
\hspace{7mm}$A$\hspace{7mm} & \hspace{5mm}$e(\Gamma,A\times C_2)$ for $\Gamma\simeq \Aut(A)$\hspace{5mm} \\
\hline
$\mathrm{M}_{12}$ & $1,584$\\
$\mathrm{M}_{22}$ & $3,432$ \\
$\mbox{HS}$ & $48,400$ \\
$\mathrm{J}_{2}$ & $3,600$ \\
$\mbox{McL}$ & $226,800$ \\
$\mbox{Suz}$ & $5,458,752$\\
$\mbox{He}$ & $533,120$ \\
$\mbox{HN}$ & $150,480,000$\\
$\mbox{Fi}_{22}$ & $83,521,152$ \\
$\mbox{Fi'}_{24}$ & $11,373,535,579,392$ \\
$\mbox{O'N}$ & $5,249,664$ \\
$\mathrm{J}_3$ & $41,040$ \\
\hline
\end{longtable}
\noindent by applying Theorem~\ref{thm1}. A double cover of $A$ exists if and only if the Schur multiplier $\mathrm{Schur}(A)$ of $A$ has order divisible by two. Among the $12$ sporadic simple groups $A$ above, it is known that 
\[\mathrm{Schur}(A)\mbox{ has even order} \iff A \simeq \mathrm{M}_{12},\mathrm{M}_{22},\mathrm{HS},\mathrm{J}_2,\mathrm{Suz},\mathrm{Fi}_{22}.\]
For these six sporadic simple groups $A$, based on \cite{Atlas}, there is no element in $\Aut(A)$ whose centralizer has order $2$ or $4$. This implies that condition (2) in Theorem~\ref{thm3} is not satisfied, so $e(\Gamma,\Delta)=0$ if $\Delta$ is a double cover of $A$. We have thus determined $e(\Gamma,\Delta)$ completely.

\section*{Acknowledgments}

Research supported by the Young Scientists Fund of the National Natural Science Foundation of China (Award No.: 11901587).

\vspace{1mm}

The author thanks the referee for helpful comments.


\begin{thebibliography}{99}

\bibitem{Byott squarefree}
A. A. Alabdali and N. P. Byott, \emph{Counting Hopf-Galois structures on cyclic field extensions of squarefree degree}, J. Algebra 493 (2018), 1--19.

\bibitem{magma}
W. Bosma, J. Cannon, and C. Playoust, \emph{The Magma algebra system. I. The user language}, J. Symbolic Comput., 24 (1997), 23--265.

\bibitem{By96}
N. P. Byott, \emph{Uniqueness of Hopf-Galois structure of separable field extensions}, Comm. Algebra 24 (1996), no. 10, 3217--3228. Corrigendum, \emph{ibid}. no. 11, 3705.

\bibitem{Byott pq}
N. P. Byott, \emph{Hopf-Galois structures on Galois field extensions of degree $pq$}, J. Pure Appl. Algebra 188 (2004), no. 1--3, 45--57.

\bibitem{Byott simple}
N. P. Byott, \emph{Hopf-Galois structures on field extensions with simple Galois groups}, Bull. London Math. Soc. 36 (2004), no. 1, 23--29.

\bibitem{Byott almost cyclic}
N. P. Byott, \emph{Hopf-Galois structures on almost cyclic field extensions of $2$-power degree}, J. Algebra 318 (2007), no. 1, 351--371.

\bibitem{Byott Childs}
N. P. Byott and L. N. Childs, \emph{Fixed-point free pairs of homomorphisms and nonabelian Hopf-Galois structures}, New York J. Math. 18 (2012), 707--731.



\bibitem{Childs simple}
S. Carnahan and L. N. Childs, \emph{Counting Hopf-Galois structures on non-abelian Galois field extensions}, J. Algebra 218 (1999), no. 1, 81--92.

\bibitem{Childs book}
L. N. Childs, \emph{Taming wild extensions: Hopf algebras and local Galois module theory}. Mathematical Surveys and Monographs, 80. American Mathematical Society, Providence, RI, 2000.



\bibitem{G book}
D. Gorenstein, \emph{Finite simple groups: An introduction to their classification}. University Series in Mathematics. Plenum Publishing Corp., New York, 1982.

\bibitem{GP}
C. Greither and B. Pareigis, \emph{Hopf-Galois theory for separable field extensions}, J. Algebra 106 (1987), no. 1, 261--290.

\bibitem{RG}
R. M. Guralnick, \emph{Subgroups of prime power index in a simple group}, J. Algebra 81 (1983), no. 2, 304--311.

\bibitem{PSL1}
R. M. Guralnick and P. H. Tiep, \emph{Low-dimensional representations of special linear groups in cross characteristics}, Proc. London Math. Soc. (3) 78 (1999), 116--138.

\bibitem{Kohl98}
T. Kohl, \emph{Classification of the Hopf-Galois structures on prime power radical extensions}, J. Algebra 207 (1998), no. 2, 525--546.

\bibitem{Kohl ANT}
T. Kohl, \emph{Hopf-Galois structures arising from groups with unique subgroup of order $p$}, Algebra Number Theory 10 (2016), no. 1, 37--59.

\bibitem{S6}
T. Y. Lam and D. B. Leep, \emph{Combinatorial structure on the automorphism group of $S_6$}, Exposition. Math. 11 (1993), no. 4, 289--308.

\bibitem{PSL2}
V. Landazuri and G. M. Seitz, \emph{On the minimal degrees of projective representations of the finite Chevalley groups}, J. Algebra 32 (1974) 418--443.

\bibitem{Truman}
P. J. Truman, \emph{Commuting Hopf-Galois structures on a separable extension}, Comm. Algebra 46 (2018), no. 4, 1420--1427.

\bibitem{Tsang HG}
C. Tsang, \emph{Non-existence of Hopf-Galois structures and bijecitive crossed homomorphisms}, J. Pure Appl. Algebra 223 (2019), no. 7, 2804--2821.

\bibitem{Tsang PAMS}
C. Tsang, \emph{Hopf-Galois structures of isomorphic type on a non-abelian characteristically simple extension}, Proc. Amer. Math. Soc. 147 (2019), no. 12, 5093--5103.

\bibitem{Tsang Sn}
C. Tsang, \emph{Hopf-Galois structures on a Galois $S_n$-extension}, J. Algebra 531 (2019), no. 1, 349--361.

\bibitem{Tsang ASG}
C. Tsang, \emph{On the multiple holomorph of a finite almost simple group}, New York J. Math. 25 (2019), 949--963.

\bibitem{Tsang solvable}
C. Tsang and C. Qin, \emph{On the solvability of regular subgroups in the holomorph of a finite solvable group}, Internat. J. Algebra Comput. 30 (2020), no. 2, 253--265. 

\bibitem{Tsang QS}
C. Tsang, Hopf-Galois structures on finite extensions with quasisimple Galois group, preprint. arXiv:2001.05718.

\bibitem{Wilson}
R. A. Wilson, \emph{The finite simple groups}. Graduate Texts in Mathematics, 251. Springer-Verlag London, Ltd., London, 2009.

\bibitem{Atlas}
R. Wilson, P. Walsh, J. Tripp, I. Suleiman, R. Parker, S. Norton, S. Nickerson, S. Linton, J. Bray, and R. Abbott, \emph{A world-wide-web \textsc{Atlas} of group representations} - Version 3. \url{http://brauer.maths.qmul.ac.ul/Atlas/vs/}

\end{thebibliography}
\end{document}